\documentclass[12pt,reqno]{amsart}
\usepackage{latexsym,amsmath,amsfonts,amscd,amssymb}
\usepackage{graphics,color}
\usepackage[utf8]{inputenc}

\usepackage{tikz}
\usepackage{dsfont}
\usetikzlibrary{matrix}
\usepackage{enumitem}

\renewcommand{\a}{\alpha}
\renewcommand{\b}{\beta}

\newcommand{\cF}{\mathcal{F}}

\newcommand{\cT}{\mathcal{T}}

\newcommand{\QQ}{\mathbb{Q}}

\newcommand{\NN}{\mathbb{N}}

\newcommand{\CC}{\mathbb{C}}

\newcommand{\ZZ}{\mathbb{Z}}

\newcommand{\bt}{\mathbf{t}}

\newcommand{\la}{\langle}
\newcommand{\ra}{\rangle}

\DeclareMathOperator{\lcm}{lcm}

\newcommand{\CP}{\CC P}

\newcommand{\orb}{\mathrm{orb}}

\theoremstyle{plain}
\newtheorem{proposition}{Proposition}

\newtheorem{theorem}[proposition]{Theorem}
\newtheorem{lemma}[proposition]{Lemma}
\newtheorem{corollary}[proposition]{Corollary}
\theoremstyle{definition}
\newtheorem{definition}[proposition]{Definition}

\theoremstyle{remark}
\newtheorem{remark}{Remark}
\newtheorem{question}{Question}

\begin{document}

\title[On the classification of Sasakian Smale-Barden manifolds]{On the classification of Smale-Barden manifolds with Sasakian structures}
\author[V. Mu\~{n}oz]{Vicente Mu\~{n}oz}
\address{Departamento de \'Algebra, Geometr\'{\i}a y Topolog\'{\i}a, Universidad de M\'alaga, 
Campus de Teatinos, s/n, 29071 M\'alaga, Spain}
\email{vicente.munoz@uma.es}
\author[A. Tralle]{Aleksy Tralle}
\address{Faculty of Mathematics and Computer Science, University of Warmia and Mazury, S\l\/oneczna 54, 10-710 Olsztyn, Poland}
\email{tralle@matman.uwm.edu.pl}
\date{\today}

\maketitle

\begin{abstract} 
Smale-Barden manifolds $M$ are classified by their second homology $H_2(M,\ZZ)$ and the Barden invariant $i(M)$. It is an important and dificult question to decide when $M$ admits a Sasakian structure in terms of these data. In this work we show  methods of doing this. In particular we realize all $M$ with $H_2(M)=\ZZ^k\oplus(\oplus_{i=1}^r\ZZ_{m_i}^{2g_i})$ and $i=0,\infty,$ provided that $k\geq 1$, $m_i\geq 2, g_i\geq 1$, $m_i$ are pairwise coprime. Using our methods we also contribute to the problem of the existence of definite Sasakian structures 
%and find new obstructions to the existence of Sasakian structures 
on rational homology spheres. % which were not known up to now. 
Also, we give a complete solution to the problem of the existence of Sasakian structures on rational 
homology spheres in the class of semi-regular Sasakian structures. 
\end{abstract}

\section{introduction}\label{sec:intro}
Consider a contact co-oriented manifold $(M,\eta)$ with a contact form $\eta$. We say that $(M,\eta)$ admits a {\it Sasakian structure} $(M,g,\xi,\eta,J)$ if:
\begin{itemize}[leftmargin=.2in]
\item there exists an endomorphism $J:TM\rightarrow TM$ such that 
 $J^2=-\operatorname{Id}+\xi\otimes\eta$,
for the Reeb vector field $\xi$ of $\eta$;
\item $J$ satisfies the conditions
$d\eta(JX,JY)=d\eta(X,Y)$,
for all vector fields $X,Y$ and $d\eta(JX,X)>0$ for all non-zero $X\in\ker\eta$;
\item the Reeb vector field $\xi$ is Killing with respect to the Riemannian metric 
$g(X,Y)=d\eta(JX,Y)+\eta(X)\eta(Y),$
\item the almost complex structure $I$ on the contact cone 
$C(M)=(M\times\mathbb{R}_{+},t^2g+dt^2)$
 defined by the formulas
$I(X)=J(X),X\in\ker \eta, I(\xi)=t{\partial\over\partial t},I\left(t{\partial\over\partial t}\right)=-\xi$,
 is integrable.
\end{itemize}

If one drops the condition of the integrability of  $I$, one obtains a {\it K-contact} structure.

The study of manifolds with Sasakian and, more generally, K-contact structures is an important subject, because it interrelates to other geometries as well as brings together several different fields of mathematics from algebraic topology through complex algebraic geometry to Riemannian manifolds with special holonomy. One of the important motivations to study Sasakian geometry is the fact that it can be used to construct Einstein metrics on odd-dimensional manifolds \cite[Chapter 11]{BG}, \cite{BGK}, \cite{K2}.

A $5$-dimensional simply connected manifold $M$ is called a {\it Smale-Barden manifold}. These manifolds are classified by their second homology group over $\ZZ$ and a {\it Barden invariant} \cite{B},\cite{S}. In greater detail,
let $M$ be a compact smooth oriented simply connected $5$-manifold. 
Let us write $H_2(M,\ZZ)$ as a direct sum of cyclic group of prime  power order
  $$
  H_2(M,\ZZ)=\ZZ^k\oplus( \bigoplus_{p,i}\ZZ_{p^i}^{c(p^i)}),
  $$
where $k=b_2(M)$. Choose this decomposition in a way that the second Stiefel-Whitney class map
  $$
  w_2: H_2(M,\ZZ)\rightarrow\ZZ_2
  $$
iz zero on all but one summand $\ZZ_{2^j}$. The value of $j$ is unique, it
is denoted by $i(M)$ and is called the Barden invariant.
The classification of the  Smale-Barden manifolds is given by the following theorem.

\begin{theorem}[{\cite[Theorem 10.2.3]{BG}}]\label{thm:sb-classification} 
Any simply connected closed $5$-manifold is diffeomorphic to one of the spaces
 $$
  M_{j;k_1,...,k_s;r}=X_j\# r M_\infty \#M_{k_1}\#\cdots\#M_{k_s} 
   $$
where the manifolds $X_{-1},X_0,X_j,X_{\infty}, M_j,M_{\infty}$ are characterized as follows: $1<k_i< \infty$, $k_1|k_2|\ldots |k_s$, and
\begin{itemize}
\item $X_{-1}=SU(3)/SO(3)$, $H_2(X_{-1},\ZZ)=\ZZ_2$, $i(X_{-1})=1$,
\item $X_0=S^5$, $H_2(X_0,\ZZ)=0$, $i(X_0)=0$,
\item $X_j$, $0<j<\infty$, $H_2(X_j,\ZZ)=\ZZ_{2^j}\oplus \ZZ_{2^j}$, $i(X_j)=j$,
\item $X_{\infty}=S^2\widetilde\times S^3$, $H_2(X_{\infty},\ZZ)=\ZZ$, $i(X_\infty)=\infty$
\item $M_k$, $1<k<\infty$, $H_2(M_k,\ZZ)=\ZZ_k\oplus \ZZ_k$, $i(M_k)=0$,
\item $M_{\infty}=S^2\times S^3$, $H_2(M_{\infty},\ZZ)=\ZZ$, $i(M_{\infty})=0$.
\end{itemize}
\end{theorem}

It is an important open problem to describe the class of Smale-Barden manifolds which admit Sasakian structures (see \cite[Chapter 10]{BG}). 
Some necessary conditions for a Smale-Barden manifold to carry a Sasakian structure are known.
 
\begin{definition} 
Let $M$ be a Smale-Barden manifold. We say that $M$ satisfies the {\it condition G-K} if the pair 
$(H_2(M,\ZZ),i(M))$ written in the form
 $$
 H_2(M,\ZZ)=\ZZ^k\oplus(\bigoplus_{p,i}(\ZZ_{p^i}^{c(p^i)}),
 $$
where $k=b_2(M)$, satisfies all of the following:
\begin{enumerate}
\item $i(M)\in\{0,\infty\}$,
\item for every prime $p$, $t(p):=\#\{i:c(p^i)>0\}\leq k+1$,
\item if $i(M)=\infty$, then $t(2):=\#\{i:c(2^i)>0\}\leq k$.
\end{enumerate}
\end{definition}

\begin{question}[{\cite[Question 10.2.1]{BG}}]  
Suppose that a Smale-Barden manifold satisfies the conditon G-K. Does it admit a Sasakian structure?
\end{question}

It is known from \cite[Corollary 10.2.11]{BG} that the condition G-K is necessary for the existence of 
K-contact, and hence Sasakian structures. However, it is not known to what extent it is sufficient. 
K\'ollar \cite{K} have found subtle obstructions to the existence of Sasakian structures on Smale-Barden manifolds with
$k=0$. Also, the recent works  \cite{CMRV} and \cite{MRT} showed the existence of \emph{homology Smale-Barden} 
and true Smale-Barden manifolds which carry a K-contact but do not carry any \emph{semi-regular} Sasakian structure
(see section \ref{sec:seifert} for definition). 
Therefore, finding sufficient conditions is an important problem.

As the condition G-K is a bound on the numbers $t(p)$
controlling the size of the torsion part at primes $p$, it is natural to start by focusing on the case of small values of $t(p)$. Let
 $$
  \bt(X)=\max\{ t(p)|\,p \text{  prime}\}.
  $$
The case of $\bt =0$ is that of the torsion-free Smale-Barden manifolds, where
we only have regular Sasakian structures (see section \ref{sec:bt=0}). The next case 
to analyse is $\bt = 1$. In this direction, we prove the following characterization of Sasakian structures with $\bt=1$.

\begin{theorem}\label{thm:main1} 
Let $M$ be a Smale-Barden manifold whose second integral homology has the form
 $$
 H_2(M,\ZZ)=\ZZ^k\oplus (\bigoplus_{i=1}^r\ZZ_{m_i}^{2g_i}),
 $$
and with $i(M)=0,\infty$.
Assume that $k\geq 1$, $m_i\geq 2, g_i\geq 1$, with $m_i$ pairwise coprime, $1\leq i\leq r$. 
Then $M$ admits a semi-regular Sasakian structure.
\end{theorem}

That is, all Smale-Barden manifols with $k\geq 1$ and $\bt=1$ admit Sasakian structures, and moreover they admit
semi-regular Sasakian structures.

In this article we fully solve the existence problem for semi-regular Sasakian structures on rational homology spheres. The general question was discussed in \cite{BG} (compare Corollary 10.2.15 therein). Here is the result.

\begin{theorem}\label{thm:sphere-semi}
Denote by $\mathcal{T}=\{\frac12 r(r-1)| r\geq 2\}$ the set of triangular numbers. 
Let $m_i\geq 2$ be  pairwise coprime, and $g_i=\frac12 (d_i-1)(d_i-2)\in \mathcal{T}$. 
Assume further that $\gcd(m_i,d_i)=1$ for all $i$.  
Let $M$ be a Smale-Barden manifold
with $H_2(M,\ZZ)= \bigoplus_{i=1}^r \ZZ_{m_i}^{2g_i}$ and spin. Then $M$ admits a semi-regular 
Sasakian structure.

Conversely, any Smale-Barden manifold with $b_2=0$ that admits a semi-regular Sasakian structure 
is of this form.
\end{theorem} 

Our approach yields new obstructions to the existence of semi-regular Sasakian structures, which are not detected by K\'ollar's work. 
Recall that by \cite{K}, for a homology sphere $M$ whose second homology is of the form
 $$
  H_2(M,\ZZ)=\bigoplus_{p,j} \ZZ_{p^j}^{c(p^j)},
  $$
the existence of a Sasakian structure implies the following restriction on the set of $c(p^j)$:  
all but at most $10$ elements of the set $\{g_i = \frac12 c(p^j)\}$ are in $\cT$. 

Over this, we get an interesting  coprimality condition between the order of the torsion and its size (exponent):
since $\gcd(d_i,m_i)=\gcd(d_i,p^j)=1$, where $c(p^j)=2g_i=(d_i-1)(d_i-2)$, we have
 $$
 c(p^j)\in  \cT^*_p:=\cT- \left\{ \frac12 (r-1)(r-2)\, |\, r\equiv 0\pmod{p}\right\}.
 $$  
This condition follows from the proof of Theorem \ref{thm:sphere-semi}.

The methods developed in this article allow us to address also  the following problem \cite[Open Problems 10.3.3 and 10.3.4]{BG}.

\begin{question}  
Which simply connected rational homology $5$-spheres admit negative Sasakian structures?
\end{question}

We construct new examples of such rational homology spheres. The definition of definite Sasakian structures will be given in Section \ref{sec:spheres}.

\begin{theorem}\label{thm:main2} 
Let $m_i\geq 2$ be pairwise coprime, and $g_i={1\over 2}(d_i-1)(d_i-2)$. 
Assume that $\gcd(m_i,d_i)=1$. Let $M$ be a Smale-Barden manifold with $H_2(M,\ZZ)
=\bigoplus\limits_{i=1}^r\ZZ_{m_i}^{2g_i}$ and spin, with the exceptions
$\ZZ_m^2, \ZZ_2^{2n}, \ZZ_3^6$. Then $M$ admits a negative Sasakian structure.
\end{theorem}  
Finally, let us mention the prerequisites used in this article. We use without further notice basic facts on topology of 4-manifolds and surfaces in them \cite{GS}. The basic facts about complex surfaces and curves on them, as well as the tools we use can be found in \cite{GH} and \cite{GS}.  

\subsection*{Acknowledgements} 
We thank Denis Auroux, Jos\'e Ignacio Cogolludo, Javier Fern\'andez de Bobadilla and Robert Gompf for useful comments and answering our questions.
The first author was partially supported by Project MINECO (Spain) PGC2018-095448-B-I00. 
The second author was supported by the National Science Center (Poland), grant NCN no. 2018/31/D/ST1/00053.
This work was partially done when the second author visited Institute des Hautes 
\'Etudes Scientifiques at Bur-sur-Yvette. His  thanks go to the Institute for the wonderful research atmosphere.

%%%%%%%%%%%%%%%%%%%%%%%%%%%%%%%%%%%%%%%%
\section{Quasi-regular and semi-regular Sasakian structures} \label{sec:seifert}
%%%%%%%%%%%%%%%%%%%%%%%%%%%%%%%%%%%%%%%%

A Sasakian structure on a compact manifold $M$ is called {\it quasi-regular} if there is a positive 
integer $\delta$ satisfying the condition that each point of $M$ has a neighbourhood $U$ such that each 
leaf for $\xi$ passes through $U$ at most $\delta$ times. 
If $\delta=1$, the structure is called regular. It is known \cite{R} that if a compact manifold admits a 
Sasakian structure, it also admits a quasi-regular one. Thus, when we are interested in existence 
questions, we may consider Sasakian structures which are quasi-regular. In the sequel we will 
assume that the notions of a Seifert bundle and a cyclic orbifold are known. One can consult 
\cite{BG} and \cite{MRT}. In particular we will need the following.

\begin{theorem}[{\cite[Theorems 7.5.1, 7.5.2]{BG}}]\label{thm:seifert} 
Let $(M,\eta,\xi,J,g)$ be a quasi-regular Sasakian manifold. Then the space of leaves $X$ of the foliation 
determined by the Reeb field $\xi$ has a natural structure of a Kähler orbifold. The projection $M\rightarrow X$ is a Seifert bundle.

Conversely, if $(X,\omega)$ is a Kähler orbifold and $M$ is the total space of the Seifert bundle determined by the class $[\omega]$, 
then $M$ admits a quasi-regular Sasakian structure.
\end{theorem}
In the same way, one can characterize quasi-regular K-contact manifolds considering symplectic orbifolds 
instead of the K\"ahler ones (see \cite[Theorems 19 and 21]{MRT}).

Consider now a somewhat stronger condition than being a general Seifert fibration carrying a quasi-regular Sasakian structure. We begin by recalling the construction of smooth orbifolds \cite{MRT}. A smooth orbifold is an orbifold that is a topological manifold. 
From now on we restrict to $\dim M=5$ and $\dim X=4$.

\begin{proposition}[{\cite[Proposition 4]{MRT}}] 
\label{prop:1}
Let $X$ be a smooth oriented $4$-manifold with embedded surfaces $D_i$ intersecting transversally, 
and integers $m_i>1$ such that $\gcd(m_i,m_j)=1$ if $D_i$ and $D_j$ have a non-empty intersection. 
Then  $X$ admits a structure of a smooth orbifold with isotropy surfaces $D_i$ of multiplicities $m_i$.
\end{proposition}

\begin{definition} 
If $M$ is a total space of the Seifert fibration whose base $X$ is a {\it smooth}  orbifold, 
we say that $M\rightarrow X$ is a {\it semi-regular} Seifert fibration.
\end{definition}    

A semi-regular Seifert fibration can be characterized more intrinsically. For each $p\in M$, the number $\delta(p)$ of times
that a nearby generic leaf cuts a (small) transversal $T(p)$ is finite. 
Then $M$ is semi-regular if $\delta(p)=\lcm(\delta(q) | q\in T(p), q\neq p)$, for all $p\in M$.

In this case, a more detailed description is available. 
Semi-regular Seifert fibrations are determined by orbit invariants and the first Chern 
class $c_1(M/X)\in H^2(X,\QQ)$ as follows. We say that an element $a$ in a free abelian group is 
{\it primitive} if it cannot be represented as $a=k\,b$ with a non-trivial $b\in A,k\in \NN$.

\begin{proposition}[{\cite[Proposition 14]{MRT}}]\label{prop:orb-inv} 
Let $X$ be an oriented $4$-manifold and $D_i\subset X$ oriented surfaces of $X$ which intersect transversally. 
Let $m_i>1$ such that $\gcd(m_i,m_j)=1$ if $D_i$ and $D_j$ intersect. Let $0<j_i<m_i$ with $\gcd(j_i,m_i)=1$ 
for every $i$. Let $0<b_i<m_i$ such that $j_ib_i\equiv 1 \pmod{m_i}$. Let $B$ be a complex line bundle over $X$. 
Then there exists a Seifert bundle $\pi: M\rightarrow X$ with orbit invariants $\{(D_i,m_i,j_i)\}$ and the first Chern class
 \begin{equation}\label{eqn:c1MX}
 c_1(M/X)=c_1(B)+\sum_i{b_i\over m_i}[D_i].
 \end{equation}
The set of all such Seifert bundles forms a principal homogeneous space under $H^2(X,\ZZ)$, where the action corresponds to changing of $B$.
\end{proposition}

We will also use the following fact.

\begin{proposition}[{\cite[Proposition 35]{K}}]\label{prop:i} 
If $\pi: M\rightarrow X$ is a Seifert fibration such that $X$ is smooth and the codimension $2$ isotropy divisors are orientable, then 
$i(M)=0,\infty$.
\end{proposition}

The homology of the total  space of a semi-regular Seifert fibration  $M\rightarrow X$ is given by the following result.

\begin{theorem}[{\cite[Theorem 16]{MRT}}] \label{thm:Kollar}
Suppose that $\pi:M\to X$ is a semi-regular Seifert bundle with isotropy surfaces $D_i$ with multiplicities $m_i$. 
Then $H_1(M,\ZZ)=0$ if and only if
 \begin{enumerate}
 \item $H_1(X,\ZZ)=0$,
 \item $H^2(X,\ZZ)\to \bigoplus H^2(D_i,\ZZ_{m_i})$ is surjective,
 \item $c_1(M/m) =m\, c_1(M)\in H^2(X,\ZZ)$ is primitive, where $m=\lcm(m_i)$.
 \end{enumerate}
 Moreover $H_2(M,\ZZ)=\ZZ^k\oplus \bigoplus \ZZ_{m_i}^{2g_i}$, $g_i$ is the genus of $D_i$, $k+1=b_2(X)$.
\end{theorem}

We will also need to calculate the fundamental group of $M$.  By \cite[Theorem 4.3.18]{BG}, we have an exact sequence 
 $$
 \pi_1(S^1)=\ZZ \to \pi_1(M) \to  \pi^{\orb}_1(X).
 $$
If $\pi_1^{\orb}(X)=1$, then $\pi_1(M)$ is abelian,
hence $\pi_1(M)=H_1(M,\ZZ)$. The first homology group can be computed via Theorem \ref{thm:Kollar}.

\medskip

In this work we construct various Seifert bundles using some configurations of smooth curves in complex algebraic surfaces. 
The calculations of the fundamental groups use the following result.
\begin{proposition} \label{prop:abelian}
If $X$ is smooth simply-connected projective surface and $D_i$ are smooth complex curves with $D_i^2>0$ and 
they intersect transversally, then 
$$\pi_1(X-(D_1\cup \ldots \cup D_r))$$
 is abelian.
\end{proposition}

\begin{proof}
 This follows from \cite{Nori}, page 306, item II, taking $E$ empty.
\end{proof}

\medskip

We finally study the second Stiefel-Whitney class of a semi-regular Seifert bundle $\pi:M\to X$. 
The formula for Seifert fibrations proved in \cite[Lemma 36]{K} says that a 
Seifert fibration with orbit invariants as in Proposition \ref{prop:orb-inv} has the total space $M$ with the 
second Stiefel-Whitney class given by the formula
 \begin{equation}\label{eqn:w2}
  w_2(M)=\pi^*w_2(X)+\sum_i(m_i-1)[E_i]
  \end{equation}
where $E_i=\pi^{-1}(D_i)$. Note also that $\pi^*[D_i]=m_i [E_i]$. 
The result in \cite[Corollary 37]{K} gives an alternative formula in terms of the orbit invariants $b_i$,
 \begin{equation}\label{eqn:w2M}
 w_2(M)=\pi^*(w_2(X)+\sum_i b_i[D_i] + c_1(B)). 
 \end{equation}
For using it, we need to compute the kernel of $\pi^*$. This is given by the following.

\begin{proposition}\label{prop:w2}
 Let $\pi:M\to X$ be a semi-regular Seifert fibration with isotropy locus and orbit invariants
 $\{(D_i,m_i,b_i)\}$, and $H_1(M,\ZZ)=0$. The mod $2$ cohomology of $X$ is $H^2(X,\ZZ_2)=\ZZ_2^{k+1}$ and 
 the mod $2$ cohomology of $M$ is 
  $$
   H^2(M,\ZZ_2)=\ZZ_2^k \oplus (\bigoplus_{m_i \text{ even}} (\ZZ_2)^{2g_i})
  $$
The map $\pi^*:H^2(X,\ZZ_2)\to H^2(M,\ZZ_2)$ has image onto the first summand $\ZZ_2^k \subset H^2(M,\ZZ_2)$.
Its kernel is:
 \begin{itemize}
 \item If all $m_i$ are odd, then $\ker\pi^*$ is one-dimensional spanned by $c_1(B)+\sum b_i [D_i]$.
 \item If $c=\#\{i\, | \, m_i$ even$\}>0$, then $\ker\pi^*$ is $c$-dimensional, and $\ker\pi^*$ is spanned by those $[D_i]$ with $m_i$ even.
 \end{itemize}
\end{proposition}

\begin{proof}
We consider the Leray spectral sequence of the Seifert fibration $\pi:M\to X$ with coefficients in $\ZZ_2$.
This says that $H^i(X,R^j\pi_*\underline{\ZZ}{}_2)\Rightarrow H^{i+j}(M,\ZZ_2)$.
By Theorem \ref{thm:Kollar}, we have the cohomology of $X$ as $H^1(X,\ZZ_2)=0$,
$H^2(X,\ZZ_2)=\ZZ_2^{k+1}$ and $H^3(X,\ZZ_2)=0$, where $k+1=b_2(X)$.

Now let $\cF=R^1\pi_*\underline{\ZZ}{}_2$. All fibers have $H^1(\pi^{-1}(x),\ZZ_2)=\ZZ_2$, so $\cF(U)=\ZZ_2$ for
any small ball $U\subset X$. If $x\in D_i$ is an isotropy set with $m_i$ odd, and $V\subset U$ then $\cF(U)\to \cF(V)$ is an isomorphism.
If $m_i$ is even then $\cF(U)\to \cF(U-D_i)$ is the zero map. Relabel the $m_i$ so that $m_1,\ldots, m_c$
are even, and denote $D=\bigsqcup_{i=1}^c D_i$. Note that the union is
disjoint. Let $j:X-D\to X$ and $i:D\to X$ denote the inclusion. Then the previous discussion means that 
 $$
  \cF= j_! \underline{\ZZ}{}_2\oplus i_* \underline{\ZZ}{}_2
  $$
Note also that by Theorem \ref{thm:Kollar}(2), reducing mod $2$, we have that $[D_1],\ldots, [D_c]$ are
linearly independent in $H^2(X,\ZZ_2)$.

Let us start first by the case where all $m_i$ are odd. Then $F=\underline{\ZZ}{}_2$, and so the Leray
spectral sequence is
\begin{center}
\begin{tikzpicture}
  \matrix (m) [matrix of math nodes,
    nodes in empty cells,nodes={minimum width=8ex,
    minimum height=5ex,outer sep=-5pt},
    column sep=1ex,row sep=1ex]{
          \ZZ_2  & 0 & \ZZ_2^{k+1} &  0  & \ZZ_2 \\ 
          \ZZ_2  & 0 & \ZZ_2^{k+1} &  0  & \ZZ_2 \\ 
   };
  \draw[-stealth] (m-1-2.south east) -- (m-2-4.north west);
  \draw[-stealth] (m-1-3.south east) -- (m-2-5.north west);
  \draw[-stealth] (m-1-1.south east) -- (m-2-3.north west);
\end{tikzpicture}
\end{center}  
The map $H^0(X,\cF)\to H^2(X,\ZZ_2)$ is given by cup product with the Chern class
$c_1(M/m)=m\, c_1(M) = c_1(B)+ \sum b_i [D_i] \pmod2$ (using that $m$ is odd).

Now assume that $c>0$. Clearly $\pi^*[D_i]=m_i[E_i]=0 \pmod2$, so $\ker\pi^*$ is at least $c$-dimensional.  
Now we compute $H^k(X,\cF)$. 
Clearly 
 $$
 H^0(X,j_*\underline{\ZZ}{}_2)=\ZZ_2^c,  H^1(X,j_*\underline{\ZZ}{}_2)=\ZZ_2^{2(g_1+\ldots+g_c)}, 
 H^2(X,j_*\underline{\ZZ}{}_2)=\ZZ_2^c
  $$
Using the exact sequence $j_!\underline{\ZZ}{}_2 \to \underline{\ZZ}{}_2\to j_*\underline{\ZZ}{}_2$, 
and that the map $H^2(X,\ZZ_2) \to \bigoplus_{i=1}^c H^2(D_i,\ZZ_2)$ is surjective,
we get
 \begin{align*}
  & H^0(X,j_!\underline{\ZZ}{}_2)=0,  H^1(X,j_!\underline{\ZZ}{}_2)=\ZZ_2^{c-1},
 H^2(X,j_!\underline{\ZZ}{}_2)=\ZZ_2^{2(g_1+\ldots+g_c)+k+1-c}, \\ 
 &  H^3(X,j_!\underline{\ZZ}{}_2)=0, H^4(X,j_!\underline{\ZZ}{}_2)=\ZZ_2
 \end{align*}
So the $E_2$ term of the Leray spectral sequence becomes
\begin{center}
\begin{tikzpicture}
  \matrix (m) [matrix of math nodes,
    nodes in empty cells,nodes={minimum width=12ex,
    minimum height=5ex,outer sep=-5pt},
    column sep=1ex,row sep=1ex]{
          \ZZ_2^c  & \ZZ_2^{2(g_1+\ldots+g_c)+c-1} & \ZZ_2^{2(g_1+\ldots+g_c)+k+1} &  0  & \ZZ_2 \\ 
          \ZZ_2  & 0 & \ZZ_2^{k+1} &  0  & \ZZ_2 \\ 
   };
  \draw[-stealth] (m-1-2.south east) -- (m-2-4.north west);
  \draw[-stealth] (m-1-3.south east) -- (m-2-5.north west);
  \draw[-stealth] (m-1-1.south east) -- (m-2-3.north west);
\end{tikzpicture}
\end{center}  
Therefore the map $\ZZ_2^c\to \ZZ_2^{k+1}$ must be injective, and the  map $\pi^*:H^2(X,\ZZ_2)\to H^2(M,\ZZ_2)$ has
image on the term $E^{2,0}_2/d_2(E^{0,1})\cong \ZZ_2^{k+1-c}$.
Note that we recover the mod $2$ cohomology of $M$, as stated.
\end{proof}

%%%%%%%%%%%%%%%%%%%%%%%%%%%%%%%%%%%
\section{Proof of  Theorem \ref{thm:main1}}
%%%%%%%%%%%%%%%%%%%%%%%%%%%%%%%%%%%

The proof of the main theorem is a corollary to the results below.

\begin{theorem}\label{thm:sums1} 
Let $m_i \geq 2$, $g_i\geq 1$, with $m_i$ pairwise coprime, $1\leq i\leq r$, and with $(m_i,g_i)\neq($even,odd$)$. 
Let $M$ be a Smale-Barden manifold with $H_2(M,\ZZ)=\ZZ \oplus \left(\bigoplus_{i=1}^r \ZZ_{m_i}^{2g_i}\right)$. 
\begin{itemize}
\item If $M$ is spin then $M$ admits a semi-regular Sasakian structure. 
\item If $i(M)=\infty$ and some $m_i$ is even then $M$
admits a semi-regular Sasakian structure.
\end{itemize}
\end{theorem}

\begin{proof}
Take the Kähler manifold $X=\CP^1 \times \CP^1$. 
The second cohomology can be written as $H^2(X,\ZZ)=\ZZ\la H_1,H_2\ra$, where $H_1,H_2$ are the classes
of the two factors. So $H_1\cdot H_2=1$. The canonical class is $K_X=-2H_1-2H_2$, so $w_2(X)=0$ and $X$ is spin. 
Now take a collection of curves $D_i\subset X$ such that
 \begin{equation}\label{eqn:Di}
  [D_i]= 2H_1+(g_i+1)H_2.
 \end{equation}
The genus of $D_i$ equals $g_i$, which can be checked with the adjunction formula: 
 $$
  K_X\cdot D_i+ D_i^2=-4-2(g_i+1)+ 4(g_i+1)=2g_i-2.
  $$
For $g_i\geq 1$ and a generic section, $D_i$ is smooth. This holds since the divisor $a_1H_1+a_2H_2$ is very
ample for $a_1,a_2\geq 1$, and using Bertini's theorem. Moreover, the divisors $D_i$ can be taken to intersect transversally by 
 Bertini's theorem. 

We put coefficients $m_i$ to each of $D_i$ by using  Proposition \ref{prop:1} in order to get a Kähler orbifold.
For the given orbit invariants $b_i$ and line bundle $B$, we have by Proposition \ref{prop:orb-inv} a Seifert bundle
 \begin{equation}\label{eqn:piMX}
 \pi:M\to X
 \end{equation}
that admits a semi-regular Sasakian structure if $c_1(M/X)$ defined by (\ref{eqn:c1MX}) is an orbifold Kähler class.
  We need to compute 
 $$
 c_1(M/X)= c_1(B)+ \sum \frac{b_i}{m_i}[D_i].
 $$
For a Kähler structure, we have $[\omega]=a_1 H_1+ a_2 H_2$, with $a_1,a_2>0$, so it is enough
to choose $c_1(B)=\b_1H_1+\b_2H_2$ in a way to get $c_1(M/X)=a_1H_1+a_2H_2$, $a_1,a_2> 0$. 

 The fundamental group of $\pi_1(X-(D_1\cup\ldots D_r))$ is abelian by Proposition \ref{prop:abelian}.
 Hence $\pi_1^{\orb}(X)$ and also $\pi_1(M)$ are abelian.
 It only remains to prove that $H_1(M,\ZZ)=0$ to know that $M$ is simply-connected. This follows from Theorem
 \ref{thm:Kollar} if $c_1(M/m)$ is primitive, where $m=\lcm(m_i)=\prod m_i$, and if 
 $$
  H^2(X, \ZZ) \to \bigoplus H^2(D_i,\ZZ_{m_i})
  $$
is surjective. As all $m_i$ are coprime, the surjectivity above is equivalent to the surjectivity of
each individual map $H^2(X, \ZZ) \to H^2(D_i,\ZZ_{m_i})$, and this happens if the divisibility of $D_i$ is coprime with $m_i$.
By (\ref{eqn:Di}) the class $[D_i]$ is primitive if $g_i$ is even, and the divisibility is $2$ if $g_i$ is odd. Therefore
we need to assume that $(g_i,m_i) \neq ($odd,even$)$.

Next we need to compute
 \begin{align*}
 c_1(M/m) &=m c_1(B)+ \sum \frac{m}{m_i} b_i [D_i] \\
  &= (m\b_1+ \sum 2M_ib_i) H_1 + (m\b_2 + \sum (g_i+1)M_i b_i) H_2,
 \end{align*}
where $M_i=m/m_i$, to be primitive. We have the following cases:
 \begin{itemize}
  \item If all $m_i$ are odd, that is $m$ is odd. Then $\gcd(2M_1,\ldots, 2M_r,m)=1$, and we can solve the equation
$m\b_1+ \sum 2M_ib_i=1$, thereby getting a primitive class. Now choosing any $\b_2$ so that 
$m\b_2 + \sum (g_i+1)M_i b_i>0$, we get an orbifold Kähler class, and the manifold $M$ in (\ref{eqn:piMX}) is 
semi-regular Sasakian with $\pi_1(M)=0$.

\item If $m$ is even, then we can assume that $m_1$ is even and $m_i$ is odd for $i>1$. Then $g_1$ is also even
by our assumption. Setting $\b_1=0$, we can solve
$\sum 2M_ib_i=2$. Now $\sum (g_i+1)M_i b_i= (\sum g_i M_ib_i )+ 1$ is
odd, so taking any $\b_2$ such that 
$m\b_2 + \sum (g_i+1)M_i b_i>0$, we get an orbifold Kähler class, and the manifold $M$ in (\ref{eqn:piMX}) is 
semi-regular Sasakian with $\pi_1(M)=0$.

We can also take $\b_1=1$ and solve $m+\sum 2M_ib_i=2$, and work to get the same result.
\end{itemize}

Now let us compute $w_2(M)$. If $m$ is odd, then (\ref{eqn:w2}) says that
$w_2(M)=\pi^*w_2(X)+\sum_i(m_i-1)[E_i]=0$ since $w_2(X)=0$, so $M$ is spin.

If $m$ is even, then arrange $m_1$ even and the others are odd. By Proposition \ref{prop:w2}, 
we have that $\pi^*:H^2(X,\ZZ_2)\to H^2(M,\ZZ_2)$ has one-dimensional kernel spanned by
$[D_1]$. As $m_1$ is even we have that $g_1$ is even, by the assumption of the statement. 
Hence $\pi^*[D_1]=\pi^*[H_2]=0$. Therefore (\ref{eqn:Di}) implies that also $\pi^*[D_j]=0$ 
for all $j\geq 1$. As $w_2(X)=0$, we have that (\ref{eqn:w2M}) implies that 
 $$
  w_2(M)=\pi^*(w_2(X)+\sum_ib_i[D_i])+\pi^*c_1(B)=\pi^*c_1(B)=\b_1 \, \pi^*[H_1].
  $$
Taking either $\b_1=0$ or $\b_1=1$ we get $M$ spin or non-spin, as desired.
\end{proof}

\begin{theorem}\label{thm:sums2}
Assume that we are given $m_i \geq 2$, $g_i\geq 1$, with $m_i$ pairwise coprime, $1\leq i\leq r$, and with $(m_i,g_i)\neq($even,even$)$. 
Let $M$ be a Smale-Barden manifold with $H_2(M,\ZZ)=\ZZ \oplus \left(\bigoplus_{i=1}^r \ZZ_{m_i}^{2g_i}\right)$. 
\begin{itemize}
\item If $M$ is non-spin then $M$ admits a semi-regular Sasakian structure. 
\item If $i(M)=0$ and some $m_i$ is even then $M$ admits a semi-regular Sasakian structure. 
\end{itemize}
\end{theorem}

\begin{proof}
Consider the projective plane $\CP^2$ blown-up at a point, $X=\CP^2\#\overline{\CP}{}^2$. This is a Kähler manifold with
$H_2(X,\ZZ)$ generated by the hyperplane class $H$ an the exceptional divisor $E$, $H^2=1$, $H\cdot E=0$, $E^2=-1$. 
Consider plane curves $D_i'$ of degree $d_i\geq 3$ which are smooth except for 
a point of multiplicity $d_i-2$ with different branches. After blowing-up, the proper transform is a
smooth curve 
 \begin{equation}\label{eqn:Di2}
 D_i\equiv d_iH-(d_i-2)E
 \end{equation}
Note that the divisor $d_i H-(d_i-2)E$ is very ample on $X$ since $H-E$ is a movable divisor 
and $2H-E$ is very ample, and adding very ample to movable
gives very ample divisors. This implies that we can choose $D_i$ to be smooth. Moreover we can
choose $D_i$ to intersect transversally. 

As $K_X=-3H+E$, we have that the genus of $D_i$ is given by
 $$
  2g_i-2=K_X\cdot D_i +D_i^2= -3d_i+d_i-2+d_i^2-(d_i-2)^2=2d_i-6,
  $$
 so $g_i=d_i-2\geq 1$.
Now put isotropy $m_i$ over $D_i$ and make it an orbifold using  Proposition \ref{prop:1}. 
Given orbit invariants $b_i$ and a line bundle $B$, we have by Proposition \ref{prop:orb-inv} a Seifert bundle
 \begin{equation}\label{eqn:piMX2}
 \pi:M\to X
 \end{equation}
that admits a semi-regular Sasakian structure if $c_1(M/X)$ defined by (\ref{eqn:c1MX}) is an orbifold Kähler class.
We compute 
 $$
 c_1(M/X)= c_1(B)+ \sum \frac{b_i}{m_i}[D_i]. 
 $$
For a Kähler structure, we have $[\omega]=a_1 H- a_2 E$, with $a_1>a_2>0$, so it is enough
to choose $c_1(B)=\b_1H_1+\b_2H_2$ with $c_1(M/X)=a_1H- a_2E$, $a_1>a_2> 0$. 

 The fundamental group of $\pi_1(X-(D_1\cup\ldots D_r))$ is abelian by Proposition \ref{prop:abelian}.
 Hence $\pi_1^{\orb}(X)$ and also $\pi_1(M)$ are abelian.
 It only remains to prove that $H_1(M,\ZZ)=0$ to know that $M$ is simply-connected. This follows from Theorem
 \ref{thm:Kollar} if $c_1(M/m)$ is primitive, where $m=\lcm(m_i)=\prod m_i$, and if 
 $$
  H^2(X, \ZZ) \to \bigoplus H^2(D_i,\ZZ_{m_i})
  $$
is surjective. As all $m_i$ are coprime, the surjectivity above is equivalent to the surjectivity of
each map $H^2(X, \ZZ) \to H^2(D_i,\ZZ_{m_i})$, and this happens if the divisibility of $D_i$ is coprime with $m_i$.
By (\ref{eqn:Di2}) the class $[D_i]$ is primitive if $g_i$ is odd since then $d_i$ is odd, and it has divisibility $2$ if $g_i$ is even. 
Therefore we need to assume that $(g_i,m_i) \neq ($even,even$)$.

Next we compute
 \begin{align*}
 c_1(M/m) &=m c_1(B)+ \sum \frac{m}{m_i} b_i [D_i] \\
  &= (m\b_1+ \sum M_ib_id_i) H - (m\b_2 + \sum M_i b_i(d_i-2)) E,
 \end{align*}
where $M_i=m/m_i$, to be primitive. Let $a_1=m\b_1+ \sum M_ib_id_i$ and $a_2=m\b_2 + \sum M_i b_i(d_i-2)$.
Note that $\gcd(a_1,a_2)=\gcd(a_1,a_2-a_1)$, where $a_2-a_1=m(\b_1-\b_2)+ \sum 2M_ib_i$.
We have the following cases:
 \begin{itemize}
  \item If all $m_i$ are odd, that is $m$ is odd. Then we can solve the equation 
  \begin{equation}\label{eqn:co-odd}
  m(\b_1-\b_2)+ \sum 2M_ib_i=1 
  \end{equation}
  thereby getting a primitive class. Now choose $\b_2$ so that $a_2=m\b_2 + \sum M_i b_i(d_i-2)>0$. 
We get an orbifold Kähler class, and the manifold $M$ in (\ref{eqn:piMX2}) is 
semi-regular Sasakian with $\pi_1(M)=0$.

\item If $m$ is even, then we can assume that $m_1$ is even and $m_i$ is odd for $i>1$. Then $g_1$ must be odd
by our assumption. We can solve 
\begin{equation}\label{eqn:co-even}
m(\b_1-\b_2)+ \sum 2M_ib_i=2.
\end{equation}
 As $m_1$ is even, it must be $b_1$ odd, hence
$a_1=m\b_1+ \sum M_ib_id_i \equiv M_1b_1d_1 \equiv 1\pmod 2$.
 Taking $\b_2$ large, we get again $a_2>0$
and hence an orbifold Kähler class. The manifold $M$ in (\ref{eqn:piMX2}) is 
semi-regular Sasakian with $\pi_1(M)=0$.

\end{itemize}

\medskip

Now let us compute $w_2(M)$. 
Suppose first that $m$ is odd, that is all $m_i$ are odd. Then (\ref{eqn:w2}) says that $w_2(M)=\pi^* w_2(X)$.
Since $K_X=-3H+E$, we have that $w_2(X)=H+E$. The images of $[D_i]$ in $H^2(X,\ZZ_2)$ are $d_iH-(d_1-2)E \equiv
d_i(H+E) \pmod 2$, hence all in the line $\la (H+E)\ra \subset H^2(X,\ZZ_2)$. By Proposition \ref{prop:w2}, the kernel
of $\pi^*:H^2(X,\ZZ_2) \to H^2(M,\ZZ_2)$ is generated by $c_1(B)+\sum b_i[D_i]$, hence $w_2(X)\in \ker \pi^*$ if
and only if $c_1(B)\in \la (H+E)\ra$. As $c_1(B)=\b_1H-\b_2E$, this is rewritten as $\b_1-\b_2\equiv 0 \pmod2$. By (\ref{eqn:co-odd})
we have that $\b_1-\b_2$ is odd, hence $w_2(X)\not\in \ker \pi^*$, $w_2(M)\neq 0$ and $M$ is non-spin.

Now assume that $m$ is even. Arrange that $m_1$ even and the others are odd. By Proposition \ref{prop:w2}, 
we have that $\pi^*:H^2(X,\ZZ_2)\to H^2(M,\ZZ_2)$ has the one-dimensional kernel spanned by
$[D_1]$. As $m_1$ is even we have that $g_1$ is odd, by the assumption of the statement, so that $[D_1]=H+E$.
Now (\ref{eqn:w2M}) says that 
 $$
 w_2(M)=\pi^*(w_2(X)+\sum_i b_i[D_i] + c_1(B))=\pi^*(\b_1 H+\b_2 E),
 $$
using that $w_2(X)=H+E\in \ker \pi^*$ and $[D_i]=d_i(H+E)\in \ker \pi^*$. Therefore $w_2(M)$ is zero or non-zero
according to the parity of $\b_1-\b_2$.
We can solve (\ref{eqn:co-even}) with $\b_1-\b_2=0$ to get $w_2(M)=0$, and we can also solve
(\ref{eqn:co-even}) with $\b_1-\b_2=1$ to get $w_2(M)\neq 0$. So $M$ can be spin or non-spin.
\end{proof}

Our last tool serves to increase the rank $k$ without modifying the torsion.

\begin{theorem}\label{thm:highrank}
If $M$ is a simply-connected Smale-Barden manifold with a semi-regular Sasakian structure, then 
$M\# X_\infty$ and $M\# M_\infty$ are also semi-regular Sasakian manifolds.
\end{theorem}

\begin{proof}
 As before, consider a Seifert bundle $\pi: M\to X$ over a  Kähler orbifold $X$. Blow-up $X$ at a point outside
 the ramification divisors to get another orbifold $\hat X$. Clearly $H_2(\hat X)=H_2(X)\oplus \ZZ\la E\ra$,
 where $E$ is the exceptional divisor. Then $K_{\hat X}=K_X+E$, and $[\hat\omega]=[\omega]-\a[E]$ is
 a Kähler class for a small $\a>0$ (we have to multiply the class $[\hat\omega]$ by a large integer number to make it integral). 
 
 Let $\{(D_i,m_i,b_i)\}$ be the orbit invariants and isotropy locus of $\pi:M\to X$. We take some large integer $N>0$
 and $\a=\frac{1}{N}$. % an small (irreducible) rational number. 
 Consider for $\hat X$ the 
 orbit invariants  $\{(D_i,m_i,N b_i)\}$  and $[\hat\omega]=N([\omega]-\a[E])=N[\omega] - [E]$. Take the 
 semi-regular Sasakian structure $\hat\pi : \hat M\to \hat X$ with the first Chern class 
 $c_1(\hat M/\hat X)=[\hat\omega]$. Hence 
  \begin{equation}\label{eqn:mmm}
 c_1(\hat M/m)= N m\,  c_1(B) - m [E] + \sum Nb_i\frac{m}{m_i}[D_i]. 
   \end{equation}
This class is primitive since $c_1(M/m)$ is primitive, and taking $\gcd(m,N)=1$.

The orbifold fundamental
group satisfies $\pi_1^{\orb}(\hat X)=\pi_1^\orb(X)$. By Theorem \ref{thm:Kollar}, $H_1(\hat M,\ZZ)=0$.
Moreover $H_2(\hat M,\ZZ)=H_2(M,\ZZ) \oplus \ZZ$.

Finally, let us compute $w_2(\hat M)$. Let $K=\ker (\pi^*:H^2(X,\ZZ_2) \to H^2(M,\ZZ_2))$ and
$\hat K=\ker (\pi^*:H^2(\hat X,\ZZ_2) \to H^2(\hat M,\ZZ_2))$. If all $m_i$ are odd, then
$K=\la c_1(B)+ \sum b_i[D_i]\ra$ and $\hat K=\la Nc_1(B)+\sum Nb_i[D_i]-  [E]\ra$, 
$w_2(M)=\pi^*w_2(X)$ and $w_2(\hat M)=\hat\pi^*w_2(\hat X)$. Also $w_2(\hat X)=w_2(X)+[E]$.
From all of this we have: 
 \begin{itemize}
 \item if $w_2(M)=0$ and $N$ odd then $w_2(\hat M)=0$; 
 \item if $w_2(M) = 0$ and $N$ even then $w_2(\hat M)\neq 0$;
\item if $w_2(M)\neq 0$ and $N$ odd then $w_2(\hat M)\neq 0$.
\end{itemize}

If some $m_i$ are even, then take $m_1,\ldots, m_c$ even and $m_i$ odd for $i>c$. Then
$K=\la [D_1],\ldots, [D_c]\ra$. Then:
 \begin{itemize}
 \item if the isotropy locus of $\hat\pi:\hat M\to \hat X$ is given by the $D_i$ as chosen above, then 
 $\hat K=\la [D_1],\ldots, [D_c]\ra$. As $w_2(\hat X)=w_2(X)+[E]$, then $w_2(\hat M)=\pi^*w_2(\hat X)\neq 0$.
 \item take as isotropy locus $D_1,\ldots, D_r$, and also the divisor $E$ with orbit invariants $m_E=2$, $b_E=1$.
 Then the Chern class (\ref{eqn:mmm}) gets modified by adding $\frac{m}2 [E]$, but it is still a Kähler class and
 primitive. The condition (2) of Theorem \ref{thm:Kollar} still holds, and the homology does not change since
 the genus $g_E=0$. Now $\hat K=\la [D_1],\ldots, [D_c],[E]\ra$
 and hence $w_2(\hat M)=0$.
\end{itemize}
\end{proof}

\noindent \emph{Proof of Theorem \ref{thm:main1}.} 
First, we do the case $k=1$. If all $m_i$ are odd, then for $i(M)=0$, we use Theorem \ref{thm:sums1},
and for $i(M)=\infty$, we use Theorem \ref{thm:sums2}.

If some of $m_i$ is even, we arrange $m_1$ even and all other $m_i$, $i>1$, odd. Then if $g_1$ is even, we use Theorem \ref{thm:sums1}
for both $i(M)=0,\infty$. Alternatively, if $g_1$ is odd then we use Theorem \ref{thm:sums2} for both $i(M)=0,\infty$.

The case $k>1$ is reduced to the case $k=1$ by using Theorem \ref{thm:highrank}.
\hfill $\Box$

\begin{remark}
Theorem \ref{thm:main1} implies that for $k\geq 1$ and $\bt=1$, the manifolds admitting 
a semi-regular Sasakian structure are the same as the manifolds admitting a semi-regular K-contact structure. This is in contrast with the examples of Smale-Barden manifolds with semi-regular K-contact structures but with no Sasakian semi-regular structures constructed in \cite{CMRV} and \cite{MRT}. It seems to be important to understand this in terms of $\bt$.
\end{remark}

%%%%%%%%%%%%%%%%%%%%%%%%%%%%%%%%%%%%%%%%%%%%
 \section{The case of homology spheres}\label{sec:spheres}
%%%%%%%%%%%%%%%%%%%%%%%%%%%%%%%%%%%%%%%%%%%%

In this section we  consider homology spheres, which correspond to the case $k=0$. 
Let 
 $$
  \cT=\left\{ \frac12 r(r-1) | r\geq 2\right\}
  $$ 
be the set of triangular numbers.
The genus of a smooth curve of degree $d$ in $\CP^2$ is $g=\frac12 (d-1)(d-2) \in \cT$.

\subsection{Proof of Theorem \ref{thm:sphere-semi}}
\begin{proof}
Start with $X=\CP^2$.
 Consider smooth curves $D_i$ of genus $g_i$ intersecting transversally. Put coefficients $m_i$ and
 make the Seifert bundle. Use Proposition \ref{prop:abelian} to check that $\pi_1(M)$ is abelian. We still need
 to check the properties of Theorem \ref{thm:Kollar}. For (2), we need $\gcd(d_i,m_i)=1$. 
 Finally, for (3), choosing $c_1(B)=\b H$ we get $c_1(M/X)=\sum \frac{b_i}{m_i}[D_i]$, so 
  $$
   c_1(M/m) = (m\b + \sum b_i M_i d_i) [H]
 $$
where $M_i=\frac{m}{m_i}$, which have $\gcd(M_1,\ldots,M_r)=1$.
Finally, we need the property
  $$
  \gcd(m,M_1d_1,\ldots, M_rd_r)=1
  $$
  to get a primitive class $c_1(M/m)$.
To see this take a prime number $p$ that $p|m$. Then
there is some $p|m_i$, hence $p\not| M_i$ and $p\not| d_i$ since $\gcd(m_i,d_i)=1$. 
So $p\not| M_id_i$.

By Proposition \ref{prop:w2}, the map $\pi^*:H^2(X,\ZZ_2)\to H^2(M,\ZZ_2)$ has kernel, hence it must be the
zero map. Therefore $w_2(X)=0$ by (\ref{eqn:w2M}).

Conversely, let $\pi: M\rightarrow X$ be a semi-regular Sasakian manifold with $b_2(M)=1$,
satisfying the assumptions of Theorem \ref{thm:Kollar}. Therefore, the isotropy locus is a collection of
complex curves $D_i\subset X$ with isotropy $m_i$. Moreover, $X$ must be simply connected and $b_2(X)=1$. 
As fake projective planes are never simply-connected, we have
$X=\CP^2$. Each $D_i$ must realize a homology class $d_iH$, where 
$H$ is the hyperplane class in $H^2(X,\ZZ)=\ZZ\langle H\rangle$. As all $D_i$ intersect,
we have that $\gcd(m_1,\ldots, m_r)=1$. 
It is well known \cite{GS} that the genus of $D_i$ is ${1\over 2}(d_i-1)(d_i-2)$, that is a triangular number.
The condition (2) of Theorem \ref{thm:Kollar} says that $H^2(X,\ZZ)\to H^2(D_i,\ZZ_{m_i})$ is surjective,
which implies that $\gcd(d_i,m_i)=1$.
\end{proof}

Now we want to compare the Smale-Barden manifolds which are rational homology spheres
admitting a semi-regular Sasakian structure with those admitting a semi-regular K-contact structure.

\begin{proposition}\label{prop:last}
Let $M$ be a Smale-Barden manifold with 
$H_2(M,\ZZ)= \bigoplus_{i=1}^r \ZZ_{m_i}^{2g_i}$ and spin, admitting a semi-regular K-contact structure. Then 
$m_i\geq 2$ are pairwise coprime, $g_i=\frac12 (d_i-1)(d_i-2)\in \cT$, and
either $\gcd(m_i,d_i)=1$ for all $i$, or $\gcd(m_i,d_i+3)=1$ for all $i$. 
\end{proposition}

\begin{proof}
The proof is analogous to the last part of the proof of Theorem \ref{thm:sphere-semi}. The difference is that
now $X$ is a symplectic $4$-manifold with $b_2=1$. As $\pi_1(M)=1$ then $\pi_1^{\orb}(X)=1$, which
implies that $\pi_1(X)=1$. The intersection form must be positive definite, hence $H_2(X,\ZZ)=\ZZ\la H\ra$,
for some integral class $H$ with $H^2=1$. The canonical class of the almost-complex structure $K_X$ satisfies
\cite[Theorem 1.4.15]{GS}
 $$
  K_X^2+\chi(X) =12.
  $$
Then $K_X^2=9$, and hence $K_X=\pm 3H$.
In the first case $D_i^2+K_X\cdot D_i=d_i^2-3d_i=2g-2$, hence $g=(d_i-1)(d_i-2)/2$. In the second case,
$D_i^2+K_X\cdot D_i=d_i^2+3d_i=2g-2$, hence $g=(d_i+1)(d_i+2)/2=(k-1)(k-2)/2$, with $k=d_i+3$. Finally it must
be $\gcd(m_i,d_i)=1$ as before.
\end{proof}

\begin{remark}
 The only symplectic $4$-manifold which is simply-connected, has $b_2=1$ and $K_X\cdot[\omega]<0$ is $\CP^2$
 by a result of Taubes \cite{Tau}. It is not known whether there are simply-connected symplectic $4$-manifolds with
 $b_2=1$ and $K_X\cdot [\omega]>0$, although it is expected that they do not exist. In such case, the possibility
 $\gcd(m_i,d_i+3)=1$ would not appear in Proposition \ref{prop:last}, and then the classification of semi-regular K-contact
 and semi-regular Sasakian Smale-Barden rational homology spheres would be the same.
\end{remark}

%%%%%%%%%%%%%%%%%%%%%%%%%%%%%%%%%%%%%%%%%%%%
\section{Definite Sasakian structures and proof of Theorem \ref{thm:main2}}
%%%%%%%%%%%%%%%%%%%%%%%%%%%%%%%%%%%%%%%%%%%%

Recall that the Reeb vector field $\xi$ on a co-oriented contact manifold $(M,\eta)$ determines a 1-dimensional foliation 
$\mathcal{F}_{\xi}$ called the {\it characteristic foliation}. If we are given a Sasakian manifold $(M,\eta,\xi,g,J)$, then 
one can define  {\it basic Chern classes} $c_k(\mathcal{F}_{\xi})$ of $\mathcal{F}_{\xi}$ which are elements of the 
basic cohomology $H^{2k}_B(\mathcal{F}_{\xi})$ (see \cite[Theorem/Definition 7.5.17]{BG}).     
We say that a Sasakian structure is positive (negative) if $c_1(\mathcal{F}_{\xi})$ can be represented by a positive 
(negative) definite $1$-form. A Sasakian structure is called null, if $c_1(\mathcal{F}_{\xi})=0$. If none of these, it is called indefinite. 

\medskip

\subsection{  Proof of Theorem \ref{thm:main2}}
The proof is a corollary to Theorem \ref{thm:sphere-semi}. To see this, one recalls the results below and makes a comparison.

\begin{theorem}[{\cite[Theorem 10.2.17]{BG}}]\label{thm:pos-sas} 
Suppose that $M$ admits a positive Sasakian structure. Then the torsion subgroup of $H_2(M,\ZZ)$ is one of the following:
 $$
 (\ZZ_m)^2, m>0, (\ZZ_5)^4, (\ZZ_4)^4,
 (\ZZ_3)^4,(\ZZ_3)^6, (\ZZ_3)^8, (\ZZ_2)^{2n},\,n>0.
 $$
\end{theorem}

\begin{theorem}[{\cite[Theorem 10.3.14]{BG}}]\label{thm:pos-sphere} 
Let $M$ be a rational homology sphere. If it admits a Sasakian structure, then it is either positive, 
and the torsion in $H_2(M,\ZZ)$ is restricted by Theorem \ref{thm:pos-sas} or it is negative. 
\end{theorem}

\begin{remark} 
The proof that rational homology spheres can carry only definite Sasakian structures actually is given in \cite[Proposition 7.5.29]{BG}. 
\end{remark}

%%%%%%%%%%%%%%%%%%%%%%%%%%%%%%%%%%%%%%%%%
\section{Regular Sasakian structures } \label{sec:bt=0}
%%%%%%%%%%%%%%%%%%%%%%%%%%%%%%%%%%%%%%%%%

The description of regular Sasakian structures is obtained in \cite[Proposition 10.4.4]{BG},  by considering 
various types of Sasakian structures on torsion-free Smale-Barden manifolds classified by 
Theorem \ref{thm:sb-classification}. For completeness we show that our methods allow us 
to describe regular Sasakian structures in a unified way (see Theorem \ref{thm:torsionfree}).
Note that if a base $X$ is a smooth manifold ($X$ is considered with an empty singular locus), then  
Theorem \ref{thm:seifert} yields a circle bundle $M\rightarrow X$ over a Kähler manifold $X$. 
This bundle is determined by the cohomology class $[\omega]\in H^2(X,\ZZ)$ as the first Chern 
class of the circle bundle. Thus, regular Sasakian structures are in one to one correspondence with such bundles.   
Let $M\rightarrow X$ be a circle bundle with Euler class $e$. Assume that $X$ is a simply 
connected closed $n$-manifold.  The class $e$ is primitive  if the map $\langle e,-\rangle: H_2(X,\ZZ)\rightarrow \ZZ$ 
is surjective, equivalently if the map %. It is straightforward to see that $e$ is indivisible if and only if the map 
$e\cup: H^{n-2}(X,\ZZ)\rightarrow H^n(X,\ZZ)\cong\ZZ$ is onto.

\begin{theorem}[{\cite[Theorem 9.12]{H}}]\label{thm:bundle-classif} 
Let $X$ be a simply connected oriented $4$-manifold and let $M\rightarrow X$ be 
a circle bundle over $X$ with primitive Euler class $e$. Then $M$ is diffeomorphic to 
\begin{itemize}
\item $M\cong\#(b_2(X)-1)S^2\times S^3$, if $M$ is spin;
\item $M\cong \#(b_2(X)-2)S^2\times S^3\#S^2\tilde{\times}S^3$, if $M$ is not spin.
\end{itemize}
\end{theorem}

The condition on $M$ to be spin is described as follows.

\begin{theorem}[{\cite[Lemma 9.10]{H}}]\label{thm:spin} 
The total space $M$ of the circle bundle determined by the primitive Euler class $e$ is spin if and only if $w_2(X)\equiv 0,e \pmod2$. 
\end{theorem}

\begin{lemma} \label{lem:1}
For a smooth projective algebraic surface $S$, the Nakai–Moishezon criterion states that 
a divisor $D$ is ample if and only if its self-intersection number $D\cdot D>0$, 
and for any irreducible curve $C$ on S we have $D\cdot C > 0$.
\end{lemma} 

\begin{corollary}
Let $X=\CP^2\# k\overline{\CP}{}^2$ be the blow-up of the projective plane at $k$ points. A divisor $D=a H-\sum b_i E_i$ is
ample if $b_i>0$ and $a>\sum b_i$.
\end{corollary}

\begin{proof}
Using Lemma \ref{lem:1} with $C=H$ and $C=E_i$, we get $a>0$, $b_i>0$. Now let  $C$ be any curve  which is not
an exceptional divisor. Then $C\equiv dH-\sum \alpha_i E_i$. Intersecting with $H$ and $E_i$ we get
$d>0$, $\alpha_i>0$. Intersecting with a line $L=H-E_i$, we get $d-\alpha_i\geq 0$.  Now
$D\cdot C=ad-\sum \alpha_i b_i>\sum b_i (d-\alpha_i)\geq 0$.
\end{proof}

\begin{theorem} \label{thm:torsionfree}
Let $M$ be a Smale-Barden manifold such that $H_2(M,\ZZ)$ has no torsion. Then $M$ always  admits  a regular 
Sasakian structure.
\end{theorem}

\begin{proof}
Take $X=\CP^2\# k\overline{\CP}{}^2$, the blow-up of the projective plane at $k$ points. Then $b_2(X)=k+1$ and the homology
of $X$ is $H_2(X)=\ZZ\la H,E_1,\ldots, E_k\ra$, where $H$ is the hyperplane class and $E_i$ are the exceptional divisors. Any 
class $e=aH-\sum b_iE_i$ is ample (that is, defined by a Kähler cohomology class)
if $a>\sum b_i$, $b_i>0$ for all $i$. And it is primitive if $\gcd(a,b_i)=1$. The canonical class of $X$ is
$K_X=-3H+\sum E_i$, and $w_2(X)\equiv K_X \pmod2$. Hence $K_X\equiv e$ if and only if all $a,b_i$ are odd numbers.

Therefore, to get spin regular Sasakian manifolds with $b_2(M)=k$ and $H_2(M,\ZZ)=\ZZ^k$, take a circle bundle $M\to X$
with $e=(2k+1)H-\sum E_i$. To get a non-spin regular Sasakian manifold, take $M\to X$ with $e= 2k H-\sum E_i$.
\end{proof}

\begin{remark}
Theorem \ref{thm:torsionfree} implies that for $\bt=0$, the manifolds admitting 
a regular Sasakian structure are the same as the manifolds admitting a regular K-contact structure.
\end{remark}

\end{document}